\documentclass[11pt]{article}
\usepackage{graphicx}

\usepackage{amsmath,amsfonts,amssymb,amsthm}
\usepackage{bbm,mathrsfs, bm, mathtools}

\title{Interaction of Poisson hyperplane processes\\ and convex bodies}
\author{Rolf Schneider}
\date{}
\newcommand{\Sd}{{\mathbb S}^{d-1}}
\newcommand{\R}{{\mathbb R}}

\newcommand{\K}{{\mathcal K}}

\newcommand{\bP}{{\mathbb P}}

\newcommand{\Rd}{{\mathbb R}^d}
\newcommand{\N}{{\mathbb N}}

\newcommand{\Ha}{\mathcal{H}}

\newcommand{\D}{{\rm d}}

\newcommand{\bE}{{\mathbb E}\,}

  \renewcommand{\exp}{{\rm exp}\,}

\newtheorem{theorem}{Theorem}
\newtheorem{lemma}{Lemma}
\newtheorem{corollary}{Corollary}

\begin{document}
\maketitle

\begin{abstract}
Given a stationary and isotropic Poisson hyperplane process and a convex body $K$ in $\R^d$, we consider the random polytope defined by the intersection of all closed halfspaces containing $K$ that are bounded by hyperplanes of the process not intersecting $K$. We investigate how well the expected mean width of this random polytope approximates the mean width of $K$ if the intensity of the hyperplane process tends to infinity.\\[1mm]
\end{abstract}

\section{Introduction}\label{sec1}

Ever since the seminal papers of R\'{e}nyi and Sulanke \cite{RS63, RS64, RS68}, the approximation of convex bodies by random polytopes has been a much-studied branch of stochastic geometry. A typical object of investigation is the convex hull of $n$ independent, identically distributed random points in a given convex body in $\R^d$. A typical question is that for the asymptotic behavior of some geometric functional of this convex hull, as the number $n$ of random points tends to infinity. Surveys, at least partially devoted to this topic, are \cite{Bar07, Bar08}, \cite{Buc85}, \cite{Hug13}, \cite{Rei10}, \cite{Sch88, Sch18}, \cite[Sect. 8.2]{SW08}, \cite{WW93}. The precise asymptotic formulas that have been obtained usually require that the convex body $K$ under consideration is either sufficiently smooth (where sometimes the existence of freely rolling balls may be sufficient) or a polytope. For general convex bodies, one has a precise asymptotic formula for the volume, denoted by $V$. Let $K\subset \R^d$ be a convex body with $V(K)=1$ (say), and let $K_n$ denote the convex hull of $n$ independent random points in $K$ with uniform distribution. Then, as shown in \cite{Schu94},
\begin{equation}\label{1.1}
\lim_{n\to\infty} n^{2/(d+1)}[1-\bE V(K_n)]=c(d) \int_{\partial K} \kappa^{1/(d+1)}\,\D\Ha^{d-1},
\end{equation}
with an explicit constant $c(d)$, where $\bE$ denotes mathematical expectation. Here $\kappa$ is the generalized Gau{\ss}--Kronecker curvature (which exists almost everywhere on $\partial K$) and $\Ha^{d-1}$ is the $(n-1)$-dimensional Hausdorff measure. However, for most convex bodies (in the sense of Baire category, see \cite{Zam80}) the right-hand side of (\ref{1.1}) is zero, so that (\ref{1.1}) gives only information on an upper bound for the order of $V(K)-\bE V(K_n)$, which should be complemented by information on a lower bound. In this sense, it was proved in \cite{BL88} that
\begin{equation}\label{1.2}
n^{-1}\ln^{d-1}n \ll V(K) -\bE V(K_n) \ll n^{-2/(d+1)}.
\end{equation}
Here the notation $f(n)\ll g(n)$ means that there exists a constant $c>0$ such that $f(n)\le cg(n)$ for all sufficiently large $n\in \N$. The constant $c$ has to be independent of $n$, but it may depend on the dimension $d$, the convex body $K$, and later on the given measure $\varphi$. For the mean width $W$, it was shown in \cite{Sch87} that
\begin{equation}\label{1.3}
n^{-2/(d+1)} \ll W(K) -\bE W(K_n) \ll n^{-1/d}.
\end{equation}
The orders are best possible; they are attained by sufficently smooth bodies on the right side of (\ref{1.2}) and the left side of (\ref{1.3}), and by polytopes on the left side of (\ref{1.2}) and the right side of (\ref{1.3}). This change of optimality makes it difficult to conjecture how a common generalization of (\ref{1.2}) and (\ref{1.3}) to general intrinsic volumes might look like. Although a guess has been formulated in \cite[p. 675]{Bar89}, this has remained one of the major mysteries in this area.

It should be mentioned that it follows from \cite{Gru83} that for most convex bodies (in the sense of Baire category) the middle terms in (\ref{1.2}) and (\ref{1.3}) oscillate, as $n\to\infty$, between the orders given by the left and right sides. More precise formulations are found in \cite[Thm. 5]{BL88} and \cite[p. 305]{Sch87}. This shows that, for general convex bodies, two-sided inequalities of type (\ref{1.2}), (\ref{1.3}) with optimal orders are the best one can expect (up to the involved constants).

Vaguely `dual' to the preceding are questions about the approximation of a convex body by the intersection of random closed halfspaces containing the body. Such questions have been treated in the plane in \cite{RS68} and in higher dimensions in \cite{BFH10}, \cite{BS10}, \cite{FHZ16}, \cite{Kal90}.

A common feature of these investigations is that a fixed number $n$ of independent random objects, points or hyperplanes, is considered, and in the end this number $n$ tends to infinity. An arguably more natural model starts with a stationary Poisson process, either of points or of hyperplanes, which is then restricted, either to the points contained in the considered convex body or to the hyperplanes not intersecting the body. The intensity of the Poisson process is finally assumed to increase to infinity. For point processes, relevant investigations are \cite{BR10}, \cite{CY14}, \cite{Par11, Par12}, \cite{Rei05}, and hyperplane processes are considered in \cite{Kal90}.

The setting in this paper consists in a stationary Poisson hyperplane process $X$ and a convex body $K$ with interior points in $\R^d$. The {\em $K$-cell} of $X$ is the random polytope defined by 
\begin{equation}\label{1.3a} 
Z_K:=\bigcap_{H\in X,\,H\cap K=\emptyset} H^-(K),
\end{equation}
where $H^-(K)$ denotes the closed halfspace bounded by $H$ that contains $K$. If the intensity of $X$ tends to infinity, the $K$-cell $Z_K$ may or may not approximate $K$, depending on the directional distribution of $X$ (an even probability measure on the unit sphere) in relation to properties of $K$. In \cite{HS14}, the approximation was measured in terms of the Hausdorff metric, and various situations of good approximation were investigated. For example, $Z_K$ converges almost surely to $K$ in the Hausdorff metric as the intensity of $X$ tends to infinity, if and only if the support of the directional distribution of $X$ contains the support of the area measure of $K$.

The majority of investigations on random approximation deals with the asymptotic behavior of geometric functionals, such as volume, mean width, number of $k$-faces, of the approximating random polytopes. In the present setting, a first result of this type was proved in \cite{Kal90}. We assume now that the stationary Poisson hyperplane process $X$ has intensity $n\in\N$, and we denote the corresponding $K$-cell by $Z_K^{(n)}$. It is assumed further that the directional distribution $\varphi$ of $X$ has a positive, continuous density with respect to spherical Lebesgue measure. Under these assumptions, Kaltenbach \cite{Kal90} proved that 
\begin{equation}\label{1.4}
n^{-2/(d+1)} \ll \bE V(Z_K^{(n)})-V(K) \ll n^{-1/d}.
\end{equation}
The proof can be considered as a `dualization' (in a non-precise sense) of that of (\ref{1.3}) and an extension to Poisson processes.

The purpose of this note is to obtain a similar counterpart to (\ref{1.2}), and thus a result of type (\ref{1.4}) with the volume replaced by the mean width $W$ (observe that under dualization, volume and mean width interchange their roles, roughly). We have to assume now that the stationary Poisson hyperplane process $X$ is also isotropic, that is, its distribution is invariant under rotations. 

\begin{theorem}\label{T1.1}
Let $X$ be a stationary and isotropic Poisson hyperplane process in $\R^d$ of intensity $n$. Let $K\subset \R^d$ be a convex body with interior points, and let $Z_K^{(n)}$ denote the $K$-cell of $X$. Then
\begin{equation}\label{1.5}
n^{-1}\ln^{d-1}n \ll \bE W(Z_K^{(n)}) - W(K) \ll  n^{-2/(d+1)}.
\end{equation}
\end{theorem}

For random polytopes generated by finitely many independent hyperplanes with a suitable distribution, depending on $K$, a similar result was proved in \cite{BS10}. Some ideas used there can be employed in the following. It turned out, however, that a proof for Poisson hyperplane processes is not straightforward and requires additional arguments. These will be presented in this note.

\section{Preliminaries}\label{sec2}

The standard scalar product of $\R^d$ is denoted by $\langle\cdot\,,\cdot\rangle$, and the induced norm by $\|\cdot\|$.  The unit ball of $\R^d$ is $B^d$, and the unit sphere is $\Sd$. Lebesgue measure on $\Rd$ is denoted by $\lambda_d$.

Hyperplanes and closed halfspaces of $\R^d$ are written in the form
$$ H(u,\tau) = \{x\in\R^d:\langle x,u\rangle=\tau\},\qquad H^-(u,\tau) = \{x\in\R^d:\langle x,u\rangle\le\tau\},$$
respectively, with $u\in\Sd$ and $\tau\in\R$. Let $\Ha$ be the space of hyperplanes in $\R^d$ with its usual topology. For a subset $M\subset\R^d$, we write
$$ \Ha_M:= \{H\in\Ha: H\cap M\not=\emptyset\}.$$

By $\K_d$ we denote the space of $d$-dimensional convex bodies (compact, convex sets with interior points) in $\R^d$. As usual, it is equipped with the Hausdorff metric, denoted by $\delta$. For $K\in\K_d$, let $R_o(K)$ be the radius of the smallest ball with center at the origin $o$ of $\R^d$ that contains $K$.

For our notation concerning point processes, we refer to \cite{SW08}, Sections 3.1 and 3.2. In particular, given a locally compact topological space $E$, we denote by $({\sf N}_s(E), {\mathcal N}_s(E))$ the measurable space of simple, locally finite counting measures on $E$. We often identify a simple counting measure $\eta\in {\sf N}_s(E)$ with its support, using $\eta(\{x\})=1$ and $x\in\eta$ synonymously. A (simple) point process in $E$ is a mapping $X:(\Omega,{\mathbf A},\bP)\to ({\sf N}_s(E), {\mathcal N}_s(E))$, where $(\Omega,{\mathbf A},\bP)$ is some some probability space, such that $\{X(C)=0\}$ is measurable  for all compact sets $C\subset E$. By $\Theta=\bE X$ we denote the intensity measure of $X$. The point process $X$ is a Poisson process if
$$ \bP(X(A)=k)=e^{-\Theta(A)} \frac{\Theta(A)^k}{k!}$$
for $k\in \N_0$ and each Borel set $A\subset E$ with $\Theta(A)<\infty$. For the independence properties of (simple) Poisson processes, we refer to \cite[Thm. 3.2.2]{SW08}. A stationary Poisson hyperplane process in $\R^d$ is a Poisson process $X$ in the space $\Ha$ of hyperplanes whose intensity measure (and hence whose distribution) is invariant under translations. The intensity measure of such a process, assumed to be $\not\equiv 0$, is of the form
$$ \Theta(A) =\gamma  \int_{\Sd}\int_{-\infty}^\infty {\mathbbm 1}_A(H(u,\tau))\,\D\tau\,\varphi(\D u)$$
for Borel sets $A\subset \Ha$. Here $\gamma>0$ is the intensity of $X$ and $\varphi$ is an even probability measure on the sphere $\Sd$, the spherical directional distribution of $X$.

We assume now that $X$ is a stationary Poisson hyperplane process in $\R^d$ of intensity $\gamma>0$ and with a non-degenerate spherical directional distribution $\varphi$. Here, `non-degenerate' means that $\varphi$ is not concentrated on any great subsphere. 

For $K\in\K_d$ we denote, as above, by $Z_K$ the $K$-cell defined by $X$ and $K$. This is a random polytope, since it is almost surely bounded. More precisely, we show the following estimate, which will later, when the intensity tends to infinity, allow us to restrict ourselves to $K$-cells contained in a sufficiently large fixed ball. 

\begin{lemma}\label{L2.1}
Let $K\in\K_d$. There are constants $a,b>0$, depending only on $\varphi$, such that
$$ \bP(R_o(Z_K) > b(R_o(K)+x)) \le 2de^{-a\gamma x}\quad\mbox{for } x\ge 0.$$
\end{lemma}

\begin{proof}
Since ${\rm supp}\,\varphi$, the support of the even measure $\varphi$, is not contained in a great subsphere, we can choose vectors $\pm e_1,\dots,\pm e_d\in {\rm supp}\,\varphi$ positively spanning $\R^d$. We can choose a sufficiently large constant $b$ and sufficiently small, pairwise disjoint neighborhoods $U_i$ of $e_i$, $i=1,\dots,2d$, such that each intersection
\begin{equation}\label{2.2}
P:= \bigcap_{i=1}^{2d} H^-(u_i,1) \quad \mbox{with }u_i\in U_i,\,i=1,\dots,2d,
\end{equation}
is a polytope with $R_o(P)\le b$. Let $x\ge 0$. If then the numbers $\tau_i$ are such that $R_o(K)\le \tau_i\le R_o(K)+x$ and if $u_i\in U_i$ for $i=1,\dots,2d$, then
$$ R_o\left(\bigcap_{i=1}^{2d} H^-(u_i,\tau_i)\right) \le b(R_o(K)+x).$$
The sets of hyperplanes
$$ A_i(x):= \{H(u,\tau): u\in U_i,\, R_o(K)\le \tau\le R_o(K)+x\},\quad i=1,\dots,2d,$$
are pairwise disjoint. If $X(A_i(x)) >0$ for $i=1,\dots,2d$, then $R_o(Z_K)\le b(R_o(K)+x)$. Therefore, observing that $\Theta(A_i(x)) = \gamma x\varphi(U_i)$  and choosing $0<a\le\varphi(U_i)$ for $i=1,\dots,2d$, we get
\begin{eqnarray*}
&&\bP(R_o(Z_K) > b(R_o(K)+x))\\
&&\le \bP( X(A_i(x))=0\mbox{ for at least one } i\in \{1,\dots,2d\})\\
&&= 1-\prod_{i=1}^{2d} (1-\bP( X(A_i(x))=0))\\
&&= 1-\prod_{i=1}^{2d} (1- e^{-\gamma\varphi(U_i)x})\\
&&\le  1-(1-e^{- \gamma ax})^{2d}\\
&&\le 2de^{-\gamma ax},
\end{eqnarray*}
by Bernoulli's inequality. This was the assertion.
\end{proof}

We shall need the consequence that the random variable $R_o(Z_K)$ is integrable. More generally:
\begin{corollary}\label{C2.1}
$R_o(Z_K)$ has finite moments of all orders.
\end{corollary}

In fact, for $k\in \N$ we have
$$ \bE R_o(Z_K)^k = \int_\Omega R_o(Z_K)^k\,\D\bP= \int_0^\infty \bP(R_o(Z_K)^k>t)\,\D t.$$
The substitution $t=[b(R_o(K)+x)]^k$ for sufficiently large $t$, together with Lemma \ref{L2.1}, shows that $\bE R_o(Z_K)^k <\infty$.

\section{Proof of the upper bound}\label{sec3}

The approach to proving the right-hand estimate of (\ref{1.5}) consists in establishing an extremal property of balls and then to find a connection to a known result on approximation of balls by convex hulls of finitely many random points. Since we are dealing with Poisson processes, this requires extra arguments in either step.

Let $X$ be a stationary Poisson hyperplane process in $\R^d$, with a nondegenerate spherical directional distribution $\varphi$ and with intensity $\gamma$. If a convex body $K\in\K_d$ is given, we denote by $Z_K$, as in  (\ref{1.3a}), the $K$-cell defined by $X$ and $K$. In order to be able to compare $Z_K$ and $Z_L$ for different $K,L\in \K_d$, we use an auxiliary Poisson process. For this, we consider the product space $E:=\Sd\times[0,\infty)$ with the product measure $\varphi\otimes \lambda_+$, where $\lambda_+$ is the Lebesgue measure on $[0,\infty)$. Let $Y$ be the Poisson process on $E$ with intensity measure $2\gamma\varphi\otimes\lambda_+$. (Its existence and uniqueness up to stochastic equivalence follows, e.g., from \cite[Thm. 3.2.1]{SW08}.) Let $M(E)$ denote the set of all locally finite subsets $S\subset E$ with the property that the set $\{u\in\Sd: (u,t)\in S \mbox{ for some }t\ge 0\}$ positively spans $\R^d$. For $\eta\in{\sf N}_s(E)$ with support in $M(E)$, we define
$$ P(\eta,K):= \bigcap_{(u,t)\in{\rm supp}\,\eta} H^-(u,h(K,u)+t)$$
for $K\in\K_d$, where $h(K,\cdot)$ denotes the support function of $K$. This is a polytope containing $K$. We shall later see that the random polytope $P(Y,K)$ is stochastically equivalent to the $K$-cell $Z_K$ defined by the hyperplane process $X$. We use the random polytopes $P(Y,K)$ to show that the function $K\mapsto \bE W(Z_K)$ is concave and continuous on $\K_d$.

\begin{lemma}\label{L3.1}
For $K,L\in\K_d$ and $\alpha\in[0,1]$,
\begin{equation}\label{3.1}
\bE W(Z_{(1-\alpha)K+\alpha L}) \ge (1-\alpha)\bE W(Z_K)+\alpha \bE W(Z_L).
\end{equation}
The functional $K\mapsto \bE W(Z_K)$ is continuous on $\K_d$.
\end{lemma}

\begin{proof}
For $\eta\in{\sf N}_s(E)$ we have
$$ (1-\alpha)P(\eta,K) +\alpha P(\eta,L) \subseteq P(\eta, (1-\alpha)K+\alpha L),$$
as follows immediately from the definition of $P(\eta,K)$ and the linearity properties of the support function. The monotonicity and linearity properties of the mean width yield
$$ W(P(\eta, (1-\alpha)K+\alpha L)) \ge (1-\alpha) W(P(\eta,K))+\alpha W(P(\eta,L)).$$
Denoting by $\bP_Y$ the distribution of $Y$, we have
$$ \bE W(P(Y,K)) = \int_{{\sf N}_s(E)} W(P(\eta,K))\,\bP_Y(\D\eta),$$
hence we can conclude that
\begin{equation}\label{3.2}
\bE W(P(Y,(1-\alpha)K+\alpha L))\ge (1-\alpha)\bE W(P(Y,K)) +\alpha\bE W(P(Y,L)).
\end{equation}

We define the Poisson hyperplane process $X_K$ by
$$ X_K(A):= X(A\setminus \Ha_{{\rm int}\,K})$$
for Borel sets $A\subset\Ha$. Its intensity measure is given by
\begin{eqnarray*}
\bE  X_K(A) &=& \Theta(A\setminus \Ha_{{\rm int}\,K})\\
&=& \gamma\int_{\Sd} \int_{-\infty}^\infty {\mathbbm 1}\{H(u,\tau)\in A\}{\mathbbm 1}\{H(u,\tau)\cap{\rm int}\,K=\emptyset\}\,\D\tau\,\varphi(\D u)\\
&=& \gamma\int_{\Sd} \int_{-\infty}^{-h(K,-u)} {\mathbbm 1}\{H(u,\tau)\in A\}\,\D\tau\,\varphi(\D u)\\
&& + \gamma\int_{\Sd} \int_{h(K,u)}^{\infty} {\mathbbm 1}\{H(u,\tau)\in A\}\,\D\tau\,\varphi(\D u)\\
&=&  2\gamma\int_{\Sd} \int_{h(K,u)}^{\infty} {\mathbbm 1}\{H(u,\tau)\in A\}\,\D\tau\,\varphi(\D u),
\end{eqnarray*}
where we have used that $\varphi$ is an even measure. Next, we define a mapping $F_K:E\to \Ha$ by
$$ F_K(u,t):= H(u,h(K,u)+t)$$
and denote for $\eta\in{\sf N}_s(E)$ by $F_K(\eta)$ the pushforward of $\eta$ under $F_K$. Then $F_K(Y)$ is a Poisson hyperplane process. For its intensity measure we obtain, for Borel sets $A\subset\Ha$,
\begin{eqnarray*}
\bE(F_K(Y))(A) &=& \bE Y(F_K^{-1}(A))\\
&=& 2\gamma \int_{\Sd} \int_0^\infty {\mathbbm 1}\{(u,t)\in F_K^{-1}(A)\}\,\D t\,\varphi(\D u)\\
&=& 2\gamma \int_{\Sd} \int_0^\infty {\mathbbm 1}\{H(u,h(K,u)+t)\in A\}\,\D t\,\varphi(\D u)\\
&=& 2\gamma \int_{\Sd} \int_{h(K,u)}^\infty {\mathbbm 1}\{H(u,\tau)\in A\}\,\D \tau\,\varphi(\D u).
\end{eqnarray*}
Thus, $X_K$ and $F_K(Y)$ have the same intensity measure. Since either of them is a Poisson process, they are stochastically equivalent. It follows that the zero cell $Z_K$ is stochastically equivalent to the random polytope $P(Y,K)$. Therefore, (\ref{3.2}) yields the assertion (\ref{3.1}).

To prepare for the continuity assertion, let $K,L\in\K_d$ and suppose, without loss of generality, that $rB^d\subset K,L$ for some $r>0$ (note that $\bE W(Z_K)$ is invariant under translations of $K$). Let $\eta\in {\sf N}_s(E)$ be such that ${\rm supp}\,\eta\in M(E)$. Let $\rho>0$ be any number such that $P(\eta,K),P(\eta,L)\subset\rho B^d$ . For the Hausdorff distance $\delta$, we state that
\begin{equation}\label{3.3}
\delta(P(\eta,K),P(\eta,L)) \le \frac{\rho}{r}\delta(K,L).
\end{equation}
In fact, setting $\delta(K,L)=:\delta$, we have
$$ L\subseteq K+\delta B^d\subseteq K+\frac{\delta}{r}K=\left(1+\frac{\delta}{r}\right)K.$$
Therefore, 
$$ h(L,u)+t\le \left(1+\frac{\delta}{r}\right)h(K,u)+t \le \left(1+\frac{\delta}{r}\right)(h(K,u)+t)$$
for $t\ge 0$, which yields
$$ H^-(u,h(L,u)+t) \subseteq\left(1+\frac{\delta}{r}\right)H^-(u,h(K,u)+t)$$
and thus
\begin{eqnarray}\label{3.3b} 
P(\eta,L) &\subseteq& \left(1+\frac{\delta}{r}\right)P(\eta,K) = P(\eta,K) +\frac{\delta}{r} P(\eta,K)\\
& \subseteq& P(\eta,K) +\frac{\delta}{r}\rho B^d.\nonumber
\end{eqnarray}
Together with the analogous inclusion with $K$ and $L$ interchanged, this gives (\ref{3.3}).

Now let $K,K_i\in\K_d$ for $i\in\N$ and suppose that $K_i\to K$ in the Hausdorff metric, as $i\to\infty$. We may assume that $rB^d\subset K$ for some $r>0$, further $\delta(K_i,K) \le 1$ and $rB^d\subset K_i$ for all $i$. For $\eta\in{\sf N}_s(E)$ with ${\rm supp}\,\eta\in M(E)$ we have, by (\ref{3.3b}),
$$ P(\eta,K_i)\subseteq\left(1+\frac{\delta(K_i,K)}{r}\right)P(\eta,K) \subseteq \left(1+\frac{1}{r}\right)R_o(P(\eta,K))B^d$$
and hence
$$ P(\eta,K_i), P(\eta,K) \subseteq R_\eta B^d\quad\mbox{with } R_\eta= \left(1+\frac{1}{r}\right) R_o(P(\eta,K)).$$
Relation (\ref{3.3}) gives
$$ \delta(P(\eta,K_i),P(\eta,K)) \le \frac{R_\eta}{r}\delta(K_i,K),$$
which implies
$$ W(P(\eta,K_i)) \to W(P(\eta,K)) \quad\mbox{as } i\to\infty.$$
It also implies
$$ W(P(\eta,K_i)) \le 2\left(1+\frac{1}{r}\left(1+\frac{1}{r}\right)\right)R_o(P(\eta,K)).$$
By Corollary \ref{C2.1},
$$ \int_{{\sf N}_s(E)}R_o(P(\eta,K))\,\bP_Y(\D \eta) = \bE R_o(P(Y,K)) = \bE R_o(Z_K) <\infty.$$
Therefore, the dominated convergence theorem yields
$$ \int_{{\sf N}_s(E)} W(P(\eta,K_i))\,\bP_Y(\D \eta) \to \int_{{\sf N}_s(E)} W(P(\eta,K))\,\bP_Y(\D \eta),$$
or equivalently $\bE W(Z_{K_i}) \to \bE W(Z_K)$ as $i\to\infty$. This proves the continuity assertion.
\end{proof}

From now on, we assume that the stationary Poisson hyperplane process $X$ is isotropic and has intensity $n$. Then its intensity measure is given by $\Theta=n\mu$ with
$$ \mu=\int_{\Sd} \int_{-\infty}^\infty {\mathbbm 1}\{H(u,\tau)\in\cdot\}\,\D\tau\,\D\sigma,$$
where $\sigma$ is the normalized spherical Lebesgue measure. For convex bodies $K,L\in\K_d$ with $K\subset L$ we have
\begin{eqnarray}\label{3.4a}
&& \int_{\Ha\setminus\Ha_K} {\mathbbm 1}\{H\cap L\not=\emptyset\}\,\mu(\D H)\nonumber\\
&&=  \int_{\Sd} \int_{-\infty}^\infty {\mathbbm 1}\{H(u,\tau)\cap L\not=\emptyset\}{\mathbbm 1}\{H(u,\tau)\cap K=\emptyset\}\,\D\tau\,\sigma(\D u)\nonumber\\
&&=  2\int_{\Sd}[h(L,u)-h(K,u)]\,\sigma(\D u)\nonumber\\
&&= W(L)-W(K).
\end{eqnarray}

\begin{lemma}\label{L3.2}
If $X$ is isotropic, then the functional 
$$ K\mapsto \frac{\bE W(Z_K)}{W(K)},\quad\K\in\K_d,$$
attains its maximum at balls.
\end{lemma}

\begin{proof}
If $X$ is isotropic, then the functional $K\mapsto \bE W(Z_K)$ is invariant under rigid motions. Since by Lemma \ref{L3.1} it is concave and continuous on $\K_d$, it is well known that on the set of convex bodies $K\in \K_d$ with given mean width $W(K)$ it attains its maximum at balls. The proof, which uses Hadwiger's `Zweites Kugelungstheorem' (\cite[pp. 170--171]{Had57}; reproduced in \cite[Thm. 3.3.5]{Sch14}), is carried out in \cite[p. 621]{BS10}.
\end{proof}

To take advantage of the preceding lemma, we connect this to a known asymptotic result about convex hulls of i.i.d. random points in a ball. First we write the result of Lemma \ref{L3.2} in the form
\begin{equation}\label{3.4}
\bE W(Z_K) - W(K) \ll \bE W(Z_{B^d})-W(B^d).
\end{equation}
We recall that here $Z_{B^d}=Z_{B^d}^{(n)}$ and that we intend to let $n$ tend to infinity. In view of this, we choose  a number $R>b$, where $b$ is the constant appearing in Lemma \ref{L2.1} for $K=B^d$, and state that
\begin{equation}\label{3.5}
\bE W(Z_{B^d}) - \bE [W(Z_{B^d}){\mathbbm 1}\{R_o(Z_{B^d})<R\}]=O(n^{-1})
\end{equation}
as $n\to \infty$ (where the constant involved in $O$ depends on $R$). For the proof, we note that the left side of (\ref{3.5}) can be estimated by
\begin{eqnarray*}
\bE [W(Z_{B^d}){\mathbbm 1}\{R_o(Z_{B^d})\ge R\}] & \le& \bE [2R_o(Z_{B^d}){\mathbbm 1}\{R_o(Z_{B^d})\ge R\}]\\
&=&2\int_\Omega R_o(Z_{B^d}){\mathbbm 1}\{R_o(Z_{B^d})\ge R\}\,\D\bP\\
&=& 2\int_0^\infty \bP(R_o(Z_{B^d}){\mathbbm 1}\{R_o(Z_{B^d}\ge R\}>t)\,\D t\\
&=& 2R\,\bP(R_o(Z_{B^d})\ge R) + 2\int_R^\infty \bP(R_o(Z_{B^d})>t)\,\D t.
\end{eqnarray*}
Lemma \ref{L2.1} provides an estimate for $\bP(R_o(Z_{B^d})\ge b(1+x))$. Inserting this for suitable values of $x$, we obtain (\ref{3.5}).

We use the bijective mapping
$$ \Delta: \Ha\setminus \Ha_{B^d} \to B^d\setminus\{o\},\quad \Delta(H(u,\tau))= \tau^{-1}u.$$
Let $\kappa_0$ be the pushforward of the measure $\mu$, restricted to $\Ha\setminus \Ha_{B^d}$, under $\Delta$. Then
\begin{equation}\label{3.5a} 
\kappa_0(A) =\frac{2}{\omega_d}\int_A \|x\|^{-(d+1)}\lambda_d(\D x)
\end{equation}
for Borel sets $A\subset B^d\setminus\{o\}$, where $\omega_d$ is the surface area of the unit sphere. The measure $\kappa_0$ is infinite, but finite on compact subsets of $B^d\setminus\{o\}$.

Let $Y_n$ denote the Poisson point process in $\R^d$ with intensity measure $n\kappa_0$. Let $Q_n$ be the convex hull of $Y_n$. Then $Q_n$ is a random polytope, which is stochastically equivalent to the polar of $Z_{B^d}$. With the constant $R>b$ chosen above, we set $r=1/R$ and $B_r= rB^d$. By (\ref{3.4a}), we have
$$ W(Z_{B^d})-W(B^d) = \int_{\Ha\setminus \Ha_{B^d}} {\mathbbm 1}\{H\cap Z_{B^d}\not=\emptyset\}\,\mu(\D H)=\kappa_0(B^d\setminus Q_n),$$
hence
$$ \bE[(W(Z_{B^d})-W(B^d)){\mathbbm 1}\{R_o(Z_{B^d})<R\}]= \bE[\kappa_0(B^d\setminus Q_n){\mathbbm 1}\{B_r\subset Q_n\}].
$$
Now it follows from (\ref{3.4}) and (\ref{3.5}) that
\begin{equation}\label{3.8}
\bE W(Z_K)-W(K) \ll \bE[\kappa_0(B^d\setminus Q_n){\mathbbm 1}\{B_r\subset Q_n\}]+O(n^{-1}).
\end{equation}

To express the latter expectation in a suitable way, we note that $Q_n$ is almost surely a simplicial polytope, hence each of its facets is the convex hull of $d$ points of $Y_n$. For any $d$ points $x_1,\dots,x_d\in Y_n$ (almost surely, they are affinely independent and their affine hull does not contain $o$), we define
$$ S(x_1,\dots,x_d):= B^d\setminus H^-(x_1,\dots,x_d),$$
where $H^-(x_1,\dots,x_d)$ is the closed halfspace bounded by ${\rm aff}\{x_1,\dots,x_d\}$ that contains $o$. Further, we define
$$ T(x_1,\dots,x_d):= S(x_1,\dots,x_d) \cap {\rm pos}\{x_1,\dots,x_d\}.$$
Then we have
\begin{eqnarray*}
&& \kappa_0(B^d\setminus Q_n){\mathbbm 1}\{B_r\subset Q_n\}\\
&& = \sum_{(x_1,\dots,x_d)\in (Y_n)^d_{\not=}} {\mathbbm 1}\{Y_n(S(x_1,\dots,x_d))=0\}\kappa_0(T(x_1,\dots,x_d)){\mathbbm 1}\{B_r\subset Q_n\},
\end{eqnarray*}
where $\eta^d_{\not=}$ denotes the set of ordered $d$-tuples of pairwise different elements from the support of $\eta$. 
We note that if $B_r\subset Q_n$, then points $x_1,\dots,x_d\in Y_n$ with $Y_n(S(x_1,\dots,x_d))=0$ automatically satisfy $x_1,\dots,x_d\in B^d\setminus B_r$ and ${\rm aff}\{x_1,\dots,x_d\}\cap B_r=\emptyset$ a.s. Therefore,
\begin{eqnarray*}
&& \kappa_0(B^d\setminus Q_n){\mathbbm 1}\{B_r\subset Q_n\}\\
&& = \sum_{(x_1,\dots,x_d)\in (Y_n)^d_{\not=}} {\mathbbm 1}\{Y_n(S(x_1,\dots,x_d))=0,\,B_r\subset Q_n\}\kappa_0(T(x_1,\dots,x_d))\\
&& \hspace{4mm}\times\,{\mathbbm 1}\{x_1,\dots,x_d\in B^d\setminus B_r\}{\mathbbm 1}\{{\rm aff}\{x_1,\dots,x_d\}\cap B_r=\emptyset\}.
\end{eqnarray*} 

Using the Slivnyak--Mecke formula (see, e.g.,  \cite[Cor. 3.2.3]{SW08}) and noting that $n\kappa_0$ is the intensity measure of $Y_n$, we obtain
\begin{eqnarray*}
&& \bE[\kappa_0(B^d\setminus Q_n){\mathbbm 1}\{B_r\subset Q_n\}]\\
&& = n^d \int_{B^d\setminus B_r}\dots\int_{B^d\setminus B_r} \bE {\mathbbm 1}\{Y_n(S(x_1,\dots,x_d))=0,\,B_r\subset {\rm conv}(Y_n\cup\{x_1,\dots,x_d\})\}\\
&& \hspace{4mm}\times\,\kappa_0(T(x_1,\dots,x_d)){\mathbbm 1}\{{\rm aff}\{x_1,\dots,x_d\}\cap B_r=\emptyset\}\,\kappa_0(\D x_1)\cdots\kappa_0(\D x_d).
\end{eqnarray*}

Let $\lambda_0:= (2/\omega_d)\lambda_d$. For Borel sets $A\subset B^d\setminus B_r$ we have 
$$\lambda_0(A)\le \kappa_0(A)\le r^{-(d+1)}\lambda_0(A).$$ 
For fixed $x_1,\dots,x_d\in B^d\setminus B_r$ with ${\rm aff}\{x_1,\dots,x_d\}\cap B_r=\emptyset$, we have
\begin{eqnarray*} 
&& \bE {\mathbbm 1}\{Y_n(S(x_1,\dots,x_d))=0,\,B_r\subset {\rm conv}(Y_n\cup\{x_1,\dots,x_d\}\}\\
&& \le  \bE {\mathbbm 1}\{Y_n(S(x_1,\dots,x_d))=0\}\\ 
&& = e^{-n\kappa_0(S(x_1,\dots,x_d))}\\
&& \le e^{-n\lambda_0(S(x_1,\dots,x_d))}.
\end{eqnarray*}

Therefore, we can estimate
\begin{eqnarray*}
&& \bE[\kappa_0(B^d\setminus Q_n){\mathbbm 1}\{B_r\subset Q_n\}]\\
&& \ll  n^d\int_{B^d}\dots\int_{B^d} e^{-n\lambda_0(S(x_1,\dots,x_d))} \,\lambda_0(T(x_1,\dots,x_d))\,\lambda_0(\D x_1)\cdots\lambda_0(\D x_d).
\end{eqnarray*}

Let $\widetilde Y_n$ be a Poisson point process in $\R^d$ with intensity measure $n\lambda_0$, and let
$$ \Pi_n:= {\rm conv}(\widetilde Y_n\cap B^d).$$ 
A similar application of the Slivnyak--Mecke formula as above yields that
\begin{eqnarray*}
&& \bE\lambda_0(B^d\setminus \Pi_n)\\
&&= n^d \int_{B^d}\dots\int_{B^d} e^{-n\lambda_0(S(x_1,\dots,x_d))}
\lambda_0(T(x_1,\dots,x_d))\,\lambda_0(\D x_1)\cdots\lambda_0(\D x_d).
\end{eqnarray*}
We conclude that
$$ \bE[\kappa_0(B^d\setminus Q_n){\mathbbm 1}\{B_r\subset Q_n\}] \ll \bE \lambda_0(B^d\setminus \Pi_n).$$
It follows from Lemma 1 in \cite{Rei05} that
$$ \bE \lambda_0(B^d\setminus\Pi_n) \ll n^{-2(d+1)}.$$
Together with (\ref{3.8}), this yields the upper bound in (\ref{1.5}).

\section{Proof of the lower bound}\label{sec4}

The proof of the left-hand estimate of (\ref{1.5}) requires only very few changes in the proof of the corresponding inequality in \cite[(1.3)]{BS10}, hence we can be brief.

For $x\in\R^d\setminus K$, we define $K^x:= {\rm conv}(K\cup \{x\})$ and set
$$ K[t]:=\{x\in\R^d: W(K^x)-W(K) \le t\}\quad\mbox{for }t>0.$$
Further, we define
$$ m(H):= \min\{W(K^x)-W(K):x\in H\}$$
for hyperplanes $H\in\Ha\setminus \Ha_K$, and
$$ \Ha_K(t) := \{H\in \Ha\setminus \Ha_K: m(M)\le t\} \quad\mbox{for }t>0.$$
It is shown in \cite{BS10} that $K[t]$ is convex and that $\Ha_K(t)$ is precisely the set of all hyperplanes that meet the convex body $K[t]$ but not $K$. 

Let $H\in\Ha\setminus\Ha_K$, and let $x_0\in H$ be such that $W(K^{x_0})-W(K)=m(H)$ (clearly, such a point exists). If no hyperplane of $X$ separates $x_0$ and $K$, then $H\cap Z_K\not=\emptyset$. It follows that $\bP(H\cap Z_K\not=\emptyset)$ is at least the probability that no hyperplane of $\Ha\setminus \Ha_K$ meets $K^{x_0}$, which is equal to $\exp[-n(W(K^{x_0})-W(K))] = e^{-nm(H)}$. Therefore, we obtain
\begin{eqnarray*}
\bE W(Z_K)-W(K) &=& \int_\Omega\int_{\Ha\setminus\Ha_K} {\mathbbm 1}\{H\cap Z_K\not=\emptyset\}\,\mu(\D H)\,\D\bP\\
&=& \int_{\Ha\setminus\Ha_K} \bP(H\cap Z_K\not=\emptyset)\,\mu(\D H)\\
&\ge& \int_{\Ha\setminus\Ha_K}e^{-nm(H)}\,\mu(\D H)\\
&\ge& \int_{\Ha\setminus \Ha_K} {\mathbbm 1}\{m(H)\le t\}e^{-nt}\,\mu(\D H)\\
&=& e^{-nt}\mu(\Ha_K(t))\\
&=& e^{-nt}[W(K[t])-W(K)].
\end{eqnarray*}
The choice $t=1/n$ gives
\begin{equation}\label{14.5.3a} 
\bE W(Z_K)-W(K) \ge e^{-1}[W(K[1/n])-W(K)].
\end{equation}
The remainder of the proof is now as in \cite[p. 619]{BS10}.

\noindent Author's address:\\[2mm]Rolf Schneider\\Mathematisches Institut, Albert-Ludwigs-Universit{\"a}t\\D-79104 Freiburg i. Br., Germany\\E-mail: rolf.schneider@math.uni-freiburg.de

\end{document}